\documentclass[11pt, twoside, letterpaper]{amsart}
\usepackage{amsmath, hyperref, color}
\usepackage{soul}
\allowdisplaybreaks

\newtheorem{thm}{Theorem}[section]

\newtheorem{lem}[thm]{Lemma}
\newtheorem{prop}[thm]{Proposition}
\theoremstyle{definition}
\newtheorem{defn}[thm]{Definition}

\theoremstyle{remark}
\newtheorem{rem}[thm]{Remark}
\newtheorem{exa}[thm]{Example}
\numberwithin{equation}{section}
\newcommand{\R}{\mathbb{R}}
\newcommand{\N}{\mathbb{N}}

\newcommand{\ind}{\mathbf{1}}
\newcommand{\cL}{\mathcal{L}}
\newcommand{\cLloc}{\mathcal{L}_{\text{loc}}}
\newcommand{\cH}{\mathcal{H}}
\newcommand{\M}{\mathcal{M}}
\newcommand{\MP}{\mathcal{M}(\PP)}
\newcommand{\MPzero}{\mathcal{M}_0(\PP)}
\newcommand{\MQ}{\mathcal{M}(\QQ)}
\newcommand{\MQzero}{\mathcal{M}_0(\QQ)}
\newcommand{\cE}{\mathcal{E}}
\newcommand{\cF}{\mathcal{F}}
\newcommand{\FF}{\mathbb{F}}
\newcommand{\EE}{\mathbb{E}}
\newcommand{\PP}{\mathbb{P}}
\newcommand{\QQ}{\mathbb{Q}}
\newcommand{\ud}{\mathrm d}

\newcommand{\Aloc}{\mathcal{A}_{\text{loc}}}
\newcommand{\Mloc}{\M_{\text{loc}}}
\newcommand{\mMloc}{\M_{\text{\emph{loc}}}}
\newcommand{\MlocP}{\M_{\text{loc}}(\PP)}
\newcommand{\MlocPzero}{\M_{0,\text{loc}}(\PP)}
\newcommand{\MlocQ}{\M_{\text{loc}}(\QQ)}
\newcommand{\MlocQzero}{\M_{0,\text{loc}}(\QQ)}
\newcommand{\mMlocP}{\M_{\text{\emph{loc}}}(\PP)}
\newcommand{\mMlocQ}{\M_{\text{\emph{loc}}}(\QQ)}
\newcommand{\mMlocQzero}{\M_{0,\text{\emph{loc}}}(\QQ)}
\newcommand{\mMlocPzero}{\M_{0,\text{\emph{loc}}}(\PP)}
\newcommand{\abs}{\ll_{{\rm loc}}}
\newcommand{\equ}{\sim_{{\rm loc}}}
\newcommand{\onedim}{{\rm dim}\,}

\newcommand{\dbra}[1]{[\kern-0.15em[ #1 ]\kern-0.15em]}
\newcommand{\dbraco}[1]{[\kern-0.15em[ #1 [\kern-0.15em[}
\newcommand{\dbraoc}[1]{]\kern-0.15em] #1 ]\kern-0.15em]}
\newcommand{\dbraoo}[1]{]\kern-0.15em] #1 [\kern-0.15em[}
\newcommand{\be}{\begin{equation}}
\newcommand{\ee}{\end{equation}}
\newcommand{\ba}{\begin{aligned}}
\newcommand{\ea}{\end{aligned}}

\begin{document}

\title[Martingale representations and absolutely continuous changes of probability]{Martingale spaces and representations \\ under absolutely continuous changes of probability}

\author[A. Aksamit]{Anna Aksamit}
\address{Anna Aksamit, School of Mathematics and Statistics, Faculty of Science, University of Sydney (Australia)}
\email{anna.aksamit@sydney.edu.au}

\author[C. Fontana]{Claudio Fontana}
\address{Claudio Fontana, Department of Mathematics ``Tullio Levi-Civita'', University of Padova (Italy)}
\email{fontana@math.unipd.it}

\thanks{The authors are grateful to Shiqi Song for suggesting the construction of the example given in Section \ref{sec:example1} and to two anonymous referees for their careful reading and valuable comments on the paper. 
This work has been supported by the European Research Council under the European Union's Seventh Framework Programme (FP7/2007-2013) / ERC grant agreement no. 335421.}

\date{\today}


\maketitle

\begin{abstract}
In a fully general setting, we study the relation between martingale spaces under two locally absolutely continuous probabilities and prove that the martingale representation property (MRP) is always stable under locally absolutely continuous changes of probability. 
Our approach relies on minimal requirements, is constructive and, as shown by a simple example, enables us to study situations which cannot be covered by the existing theory. 
\end{abstract}

\section{Introduction}

Martingale representation results have fundamental applications in stochastic control, filtering, backward stochastic differential equations and mathematical finance, notably in connection with the property of market completeness. In all these fields, absolutely continuous changes of probability play an equally important role, often leading to a substantial simplification of the problem under consideration.
This motivates the interest of studying how spaces of martingales under two absolutely continuous probabilities are connected and, more specifically, the behavior of the {\em martingale representation property} (MRP, see Definition \ref{df:mrp} below) under absolutely continuous (not necessarily equivalent) changes of probability.
In this paper, we aim at developing a general theory for these questions under minimal assumptions. This enables us to simplify and extend previous results to full generality, covering situations that cannot be addressed by the existing theory.

To the best of our knowledge, the most general result available in the literature on the behavior of the MRP under absolutely continuous changes of probability can be found in \cite[Theorem 13.12]{HWY} and can be stated as follows (see also \cite[Lemma 2.5]{JS15} and \cite[Theorem 15.2.8]{CohenElliott}):
{\em Let $\PP$ and $\QQ$ be two probability measures on $(\Omega,\cF,\FF)$ such that $\QQ\ll\PP$, with density process $Z$, and let $X=(X_t)_{t\geq0}$ be a real-valued $\PP$-local martingale having the MRP under $\PP$. Suppose that the process $[X,Z]$ has locally integrable variation under $\PP$. Then, the process $X':=X-(Z_-)^{-1}\cdot\langle X,Z\rangle^{\PP}$ is a $\QQ$-local martingale and has the MRP under $\QQ$.}
A multi-dimensional version of this result, under the additional assumption of local boundedness of $1/Z$ under $\PP$, was first obtained in \cite{Duffie85}.

The crucial assumption in the above result is the requirement that $[X,Z]$ has locally integrable variation under $\PP$ (or, equivalently, that $X$ is a special semimartingale under $\QQ$). 
This leaves open the question of whether, in the absence of such a condition, the MRP is preserved or not under an absolutely continuous change of probability. We provide a positive answer in full generality, without any further assumption beyond local absolute continuity  (Theorem \ref{thm2}). 
One of the key steps in our approach consists in replacing the usual version of Girsanov's theorem (see, e.g., \cite[Theorem III.3.11]{JS03}) with its most general version proven in \cite{Leng77}.
Besides the greater generality, our proofs are more elementary and constructive than those in \cite{Duffie85,HWY} and yield an explicit description of the stochastic integral representation (Remark \ref{rem:integrand}).
As shown by means of an explicit example (Section \ref{sec:example1}), there exist simple situations that are not covered by the existing theory and for which our results yield an explicit MRP.

From a more abstract standpoint, we obtain a new and general characterization of the set of $\QQ$-martingales as the smallest stable subset generated by suitable transformations of $\PP$-martingales (Theorem \ref{thm1}).
By relying on our main results, we then address further issues, including the practically relevant case of locally equivalent probabilities and the dimension of martingale spaces under locally absolutely continuous probabilities (Section \ref{sec:corollaries}).
In particular, these results enable us to provide a general solution to an open problem stated in \cite{Zit06}.
We want to point out that, even though the present paper focuses on theoretical aspects, our results have relevant applications, notably in mathematical finance in the context of equilibrium models (see, e.g., \cite{KP17,Zit06}).\footnote{We emphasize that, in equilibrium models (see, e.g., \cite{Zit06}), the probability measure $\QQ$ is constructed endogenously. Therefore, it is crucial to have MRP stability results which do not impose a priori conditions on the density process $Z$, unlike the existing results on MRP under changes of probability.}

The paper is structured as follows. Section \ref{sec:notation} introduces  necessary notations and terminology. Section \ref{sec:setting} recalls the setting and a crucial preliminary result due to \cite{Leng77}. Section \ref{sec:main} contains our main results, while further properties and ramifications are presented in Section \ref{sec:corollaries}. In Section \ref{sec:example}, we give some examples, including a simple one which falls beyond the scope of the existing results and to which our theory applies. The proofs of all results are collected in Section \ref{sec:proofs}.


\subsection{Notation}	\label{sec:notation}

Throughout the paper, we shall make use of the following notation, referring to \cite{JS03} for all unexplained notions. Let $(\Omega,\cF,\PP)$ be a probability space endowed with a right-continuous (not necessarily complete) filtration $\FF=(\cF_t)_{t\geq0}$.
On $(\Omega,\cF,\FF,\PP)$, we denote by $\MP$ ($\MlocP$, resp.) the set of all real-valued martingales (local martingales, resp.), tacitly assumed to have a.s. c\`adl\`ag paths. 
We let $\Aloc(\PP)$ be the set of all real-valued adapted processes of locally integrable variation and, for $A\in\Aloc(\PP)$, we denote by $A^{p,\PP}$ the dual predictable projection of $A$ under $\PP$.
The set of $\cH^1$-martingales on $(\Omega,\cF,\FF,\PP)$ is defined as $\cH^1(\PP):=\{M\in\MP:\EE[\sup_{t\geq0}|M_t|]<+\infty\}$. 
Let us also introduce $\MPzero:=\{M\in\MP : M_0=0\}$ and similarly for $\MlocPzero$ and $\cH^1_0(\PP)$.
We recall that, for every $M\in\MlocP$, there exists a unique decomposition $M=M^c+M^d$ into a continuous and a purely discontinuous local martingale.
We denote respectively by $\Mloc^c(\PP)$ and $\Mloc^d(\PP)$ the set of all real-valued continuous and purely discontinuous local martingales on $(\Omega,\cF,\FF,\PP)$.
If $M=(M_t)_{t\geq0}$ is an $\R^d$-valued process such that $M^i\in\MlocP$, for each $i=1,\ldots,d$, we denote by $L_m(M,\PP)$ the set of all $\R^d$-valued predictable processes which are integrable with respect to $M$ under the measure $\PP$ in the sense of local martingales (see \cite[Definition 2.46]{J79}).
For $H\in L_m(M,\PP)$,  the stochastic integral of $H$ with respect to $M$ is denoted by $(H\cdot M)_t:=\int_{(0,t]}H_u\ud M_u$, for all $t\geq0$, with $(H\cdot M)_0=0$, similarly as in \cite{JS03}.
Finally, for a set $\mathcal{Y}\subseteq\MlocP$, we denote by $\cL^1(\mathcal{Y},\PP)$ the stable subspace generated by $\mathcal{Y}$ in $\cH^1(\PP)$, i.e., the smallest stable subspace of $\cH^1(\PP)$ containing $\{H\cdot Y : Y\in\mathcal{Y}, H\in L_m(Y,\PP)\text{ and }H\cdot Y\in\cH^1(\PP)\}$ (see \cite[Definition 4.4]{J79}).
The class $\cLloc^1(\mathcal{Y},\PP)$ is defined in the usual way by localization.

A probability measure $\QQ$ on $(\Omega,\cF)$ is said to be {\em locally absolutely continuous} with respect to $\PP$, denoted as $\QQ\abs\PP$, if $\QQ|_{\cF_t}\ll\PP|_{\cF_t}$ for all $t\geq0$. 

\begin{rem}[On the completeness of $\FF$]		\label{rem:incomplete}
In the present paper, we shall be interested in locally absolutely continuous changes of probability from $\PP$ to $\QQ$. In particular, it may happen that $\QQ\abs\PP$, while $\QQ\ll\PP$ does not hold on $\cF_{\infty-}:=\bigvee_{t\in\R_+}\cF_t$. This implies that a $\QQ$-complete filtration $\FF$ is not necessarily $\PP$-complete (and viceversa). For this reason, we shall not assume completeness of the filtration.
Most of the standard results of stochastic calculus can be developed without relying on the completeness assumption, as long as path properties are required to hold in an a.s. sense. We refer to \cite{J79,JS03} for two complete presentations of the theory which avoid the use of complete filtrations as far as possible (see also \cite[Appendix A]{PR15}). In the following, we shall point out explicitly where the potential incompleteness of $\FF$ requires modifications of existing results.
\end{rem}

\section{Results}

\subsection{Setting and preliminaries}	\label{sec:setting}

We consider a probability space $(\Omega,\cF,\PP)$ endowed with a right-continuous (not necessarily $\PP$-complete) filtration $\FF=(\cF_t)_{t\geq0}$ and a probability measure $\QQ\abs\PP$. 
In view of \cite[Theorem III.3.4]{JS03}, the density process of $\QQ$ relative to $\PP$ is the unique non-negative process $Z\in\MP$ such that $\ud\QQ|_{\cF_t}=Z_t\,\ud\PP|_{\cF_t}$, for all $t\geq0$. 
Let us define the stopping times
\begin{equation}
\label{stimes}
\zeta := \inf\{t\in\R_+ : Z_{t-}=0\text{ or }Z_t=0\}
\quad\text{ and }\quad
\eta := \zeta\ind_{\Lambda} + \infty\ind_{\Omega\setminus\Lambda},
\;\text{ with }
\Lambda:=\{\zeta<+\infty,Z_{\zeta-}>0\}.
\end{equation}
Note that $\QQ(\zeta<+\infty)=0$, while $\QQ|_{\cF_t}\sim\PP|_{\cF_t}$ holds if and only if $\PP(Z_t>0)=1$.

The behavior of local martingales under locally absolutely continuous, but not necessarily equivalent, changes of probability has been studied in \cite{Leng77}, from which we recall the following fundamental result (compare also with \cite[Theorems 12.12 and 12.20]{HWY}).\footnote{A careful examination of the proof of \cite[Theorem 3]{Leng77} shows that it still holds true for non-complete filtrations, since it is based on standard operations in stochastic calculus which are valid in general filtrations by  \cite{J79,JS03}. In the statement of Proposition \ref{prop:Leng}, we consider without loss of generality a version of $X$ that is measurable with respect to the optional $\sigma$-field on the $\overline{\PP}$-completion of the filtration $\FF$, where $\overline{\PP}:=(\PP+\QQ)/2$. Note that, as a consequence of Lemma \ref{lem:RC}, all processes have a.s. c\`adl\`ag paths under the respective probability measures (recall that we tacitly assume that all local martingales have a.s. c\`adl\`ag paths).
}

\begin{prop}	\label{prop:Leng}
For an adapted process $X$, the following hold:
\begin{enumerate}
\item[(i)] $X\in\MQ$ if and only if $ZX\in\MP$;
\item[(ii)] $X\in\mMlocQ$ if and only if there exists a sequence of stopping times $(\tau_n)_{n\in\N}$ increasing $\QQ$-a.s. to infinity such that $(ZX)^{\tau_n}\in\mMlocP$, for each $n\in\N$;
\item[(iii)] if $X\in\mMlocP$, then
\be	\label{eq:transform}
\widehat{X} := X - \frac{1}{Z}\cdot [X,Z] + \left(\Delta X_{\eta}\ind_{\dbraco{\eta,+\infty}}\right)^{p,\PP} \in \mMlocQ.
\ee
\end{enumerate}
\end{prop}

As mentioned in the introduction, part (iii) of the above proposition represents the most general formulation of Girsanov's theorem. In particular, unlike the usual version of Girsanov's theorem (see \cite[Theorem III.3.11]{JS03}), it does not rely on the assumption $[X,Z]\in\Aloc(\PP)$.
For a generic element $M\in\MlocP$, we denote by $\widehat{M}$ the element of $\MlocQ$ defined via the right-hand side of \eqref{eq:transform}, to which we refer as the {\em Lenglart transformation} of $M$. We use an analogous notation in the case of vector-valued processes.
Similarly, for a set $\mathcal{Y}\subseteq\MlocP$, we let $\widehat{\mathcal{Y}}:=\{\widehat{Y} : Y\in\mathcal{Y}\}$.

\subsection{Main results}	\label{sec:main}

Our first main result provides a characterization of the set of $\cH^1$-martingales under $\QQ$ as the stable subspace generated by $\widehat{\MP}$ in $\cH^1(\QQ)$.

\begin{thm}	\label{thm1}
$\cH^1_0(\QQ)=\cL^1(\widehat{\MP},\QQ)$.
As a consequence, it holds that $\mMlocQzero=\cL_{{\rm loc}}^1(\widehat{\MP},\QQ)$.
\end{thm}

The above theorem shows that all $\QQ$-local martingales are generated by stochastic integrals of elements $\widehat{M}$, with $M\in\MP$. Loosely speaking, we can say that $\QQ$-martingales correspond to Lenglart transformations of $\PP$-martingales.
We want to emphasize that, despite the generality of the statement, the proof relies on rather basic facts of stochastic calculus, notably integration by parts and It\^o's formula (see Section \ref{sec:proofs_main}). 

Theorem \ref{thm1} does not assume any structure on the filtered probability space $(\Omega,\cF,\FF,\PP)$. An especially important case is when all $\PP$-local martingales can be represented as stochastic integrals of some fixed $\PP$-local martingale. 
More precisely, let us formulate the following definition.

\begin{defn}
\label{df:mrp}
We say that an $\R^d$-valued $\PP$-local martingale $X$ has the {\em martingale representation property} (MRP) under $\PP$ if $\mMlocPzero=\{H\cdot X : H\in L_m(X,\PP)\}$.
\end{defn}
Our second main result asserts the stability of the MRP under locally absolutely continuous changes of probability in its most general form, without any further assumption.

\begin{thm}	\label{thm2}
Suppose that there exists an $\R^d$-valued local martingale $X$ on $(\Omega,\cF,\FF,\PP)$ having the MRP under $\PP$. Then the process $\widehat{X}$ has the MRP under $\QQ$.
\end{thm}

\begin{rem}[Explicit MRP under $\QQ$]	\label{rem:integrand}
Theorem \ref{thm2} is proved in Section \ref{sec:proofs_main} as a direct consequence of Theorem \ref{thm1}. However, Theorem \ref{thm2} also admits a constructive proof, which provides an explicit description of the integrand appearing in the stochastic integral representation under $\QQ$. 
To this effect, let $X$ be an $\R^d$-valued local martingale  on $(\Omega,\cF,\FF,\PP)$ having the MRP under $\PP$ and let $N$ be an arbitrary element of $\MlocQzero$.
By Proposition \ref{prop:Leng}-(ii), there exists a sequence of stopping times $(\tau_n)_{n\in\N}$ increasing $\QQ$-a.s. to infinity such that $(ZN)^{\tau_n}\in\MlocPzero$, for each $n\in\N$.
Since $X$ has the MRP under $\PP$, there exist $H\in L_m(X,\PP)$ and $K^n\in L_m(X,\PP)$, for each $n\in\N$, such that
\be	\label{eq:MRP_Z}
Z=Z_0+H\cdot X
\qquad\text{and}\qquad
(ZN)^{\tau_n}=K^n\cdot X,
\qquad\text{ for every }n\in\N.
\ee
As shown in Section \ref{sec:proof_rem}, the integrand $\phi\in L_m(\widehat{X},\QQ)$ appearing in the stochastic integral representation $N=\phi\cdot\widehat{X}$ can be explicitly described as 
\be	\label{eq:integrand}\qquad
\phi:=\sum_{n\in\N}\phi^n\ind_{\dbraoc{\tau_{n-1},\tau_n}},
\qquad\text{ where }\;
\phi^n := Z_-^{-1}\left(K^n-N_-H\right)\ind_{\dbra{0,\tau_n}},
\quad\text{ for every }n\in\N,
\ee
where $\tau_0:=0$.
Note that each process $\phi^n$ is well-defined under $\QQ$, since $\QQ(\zeta<+\infty)=0$.
In particular, \eqref{eq:integrand} shows that the integrand appearing in the representation of $N$ under $\QQ$ is fully determined by the integrands appearing in the representations of $Z$ and $(ZN)^{\tau_n}$, for $n\in\N$, under $\PP$.
\end{rem}

\subsection{Further properties and results}	\label{sec:corollaries}

In this section, we present some further results and special cases of interest which can be obtained from the results stated in Section \ref{sec:main}.

\subsubsection{MRP and strongly orthogonal local martingales}	\label{sec:orth}

In martingale representation results, it is typically of interest to establish the  representation property with respect to a family of orthogonal local martingales. To this effect, let us introduce some terminology.
Given $M,N\in\MlocP$, we say that $M$ and $N$ are {\em strongly orthogonal} if $[M,N]\equiv0$ up to a $\PP$-evanescent set (in particular, this implies that $M$ and $N$ are orthogonal in the usual sense of local martingales, i.e., $MN\in\MlocP$, see \cite[Definition I.4.11]{JS03}). If $X$ is an $\R^d$-valued local martingale on $(\Omega,\cF,\FF,\PP)$, we say that it has {\em strongly orthogonal components} if $X^i$ and $X^j$ are strongly orthogonal, for all $i,j=1,\ldots,d$ with $i\neq j$. Under $\QQ$, the notion of strong orthogonality is defined in an analogous way.

In general, if the MRP holds under $\PP$ with respect to a family of strongly orthogonal local martingales, Theorem \ref{thm2} does not ensure that the same holds true under $\QQ$ (see Example \ref{ce:stong} for an explicit counterexample). 
The following proposition provides a sufficient condition for this to hold. 
As a preliminary, let $\eta=\eta^{{\rm ac}}\wedge\eta^{{\rm in}}$ be the unique decomposition of the stopping time $\eta$ into an accessible time $\eta^{{\rm ac}}$ and a totally inaccessible time $\eta^{{\rm in}}$ (see \cite[Theorem I.2.22]{JS03}).

\begin{prop}	\label{prop:orth}
Let $X$ be an $\R^d$-valued local martingale on $(\Omega,\cF,\FF,\PP)$ with strongly orthogonal components under $\PP$.
Assume furthermore that $\Delta X_{\eta^{{\rm ac}}}=0$ $\PP$-a.s. on $\{\eta^{{\rm ac}}<+\infty\}$.
Then the process $\widehat{X}$ has strongly orthogonal components under $\QQ$.
\end{prop}

In particular, the assumption that $\Delta X_{\eta^{{\rm ac}}}=0$ $\PP$-a.s. on $\{\eta^{{\rm ac}}<+\infty\}$ always holds in the following cases:
\begin{enumerate}
\item[(i)]
if the set $\{\Delta X\neq0\}\cap\{Z=0<Z_-\}$ is $\PP$-evanescent;
\item[(ii)] 
if $\QQ\equ\PP$, in which case $\PP(\eta=+\infty)=1$;
\item[(iii)] 
if the process $X=(X_t)_{t\geq0}$ is $\PP$-a.s. quasi-left-continuous.
\end{enumerate}

\begin{rem}[An open question of \cite{Zit06}]	\label{rem:Zit}
In case (ii), Proposition \ref{prop:orth} gives a complete answer to an open question formulated in \cite[Remark 2.3]{Zit06}, namely whether the MRP with respect to a local martingale having strongly orthogonal components ({\em finite representation property}, in the terminology of \cite{Zit06}) is stable under equivalent changes of probability.  Proposition \ref{prop:orth}, together with Theorem \ref{thm2}, shows that the answer is always positive, even for locally equivalent changes of probability and without any further assumption on the density process $Z$.
\end{rem}

\begin{rem}
In mathematical finance, if $X$ represents the discounted price process of a set of traded assets, the condition that $\{\Delta X\neq0\}\cap\{Z=0<Z_-\}$ is $\PP$-evanescent appearing in (i) above plays a crucial role in the study of the no-arbitrage properties of $X$ under $\QQ$, see \cite{ACJ15,F14}.
\end{rem}

In the continuous case, there is no distinction between strong orthogonality and orthogonality in the usual sense of local martingales. Hence, as an immediate consequence of Proposition \ref{prop:orth}, we deduce that orthogonality is always preserved under arbitrary absolutely continuous changes of probability for continuous local martingales.
The distinction between the two notions of orthogonality appears in the case of discontinuous local martingales.
In this case, motivated by Proposition \ref{prop:orth}, one may wonder whether orthogonality in the usual sense of local martingales is in general preserved under locally absolutely continuous changes of probability. 
As shown by Example \ref{example:usual_orth}, the answer is negative, even for equivalent changes of probability.

\subsubsection{Locally equivalent changes of probability}

We now consider the special case where the two probability measures $\QQ$ and $\PP$ are locally equivalent, corresponding to the case $\PP(\zeta<+\infty)=0$. In this case, it obviously holds that $\PP(\eta<+\infty)=0$ and Proposition \ref{prop:Leng}-(iii) yields that, for any $M\in\MlocP$, the process $\widehat{M}:=M-Z^{-1}\cdot[M,Z]$ is an element of $\MlocQ$.
In this context, we can establish the following proposition, which relies on the symmetric role of $\QQ$ and $\PP$. 

\begin{prop}	\label{prop:equiv}
Suppose that $\QQ\equ\PP$ and let $X$ be an $\R^d$-valued local martingale on $(\Omega,\cF,\FF,\PP)$. Then $X$ has the MRP under $\PP$ if and only if $\widehat{X}$ has the MRP under $\QQ$.
\end{prop}

Under the slightly stronger assumption that $\QQ\sim\PP$, a version of Proposition \ref{prop:equiv} has been recently established in \cite{KP17}. 
Note also that, in the case $\QQ\sim\PP$, the decomposition $\widehat{M}=M-(Z)^{-1}\cdot[M,Z]$ corresponds to the version of Girsanov's theorem presented in \cite{Meyer76}.

\subsubsection{Dimension of $\cH^1$-martingale spaces}

In this subsection, we study how the dimension of the martingale space $\cH^1$ behaves under locally absolutely continuous changes of probability. In particular, Proposition \ref{prop:equiv} enables us to prove the invariance of the dimension with respect to locally equivalent changes of probability. 
In line with \cite[Definition 4.38]{J79}, let us recall that an $\R^d$-valued local martingale $X$ on $(\Omega,\cF,\FF,\PP)$ is said to be a {\em basis} for $\cH^1(\PP)$ if $\cL^1(X,\PP)=\cH^1_0(\PP)$ and there exists no $\R^m$-valued local martingale $Y$ on $(\Omega,\cF,\FF,\PP)$, with $m<d$, such that $\cL^1(Y,\PP)=\cH^1_0(\PP)$. In this case, $d$ is said to be the {\em dimension} of $\cH^1(\PP)$, denoted as $\onedim\cH^1(\PP)$.
This is also closely related to the notion of {\em martingale multiplicity} introduced in \cite{DV74}. Under $\QQ$, the notions of basis and dimension are defined in an analogous way.\footnote{We write $\onedim\cH^1(\PP)=+\infty$ if it does not exist a finite-dimensional basis for $\cH^1(\PP)$, and analogously under $\QQ$.} 

\begin{prop}	\label{prop:dim}
If $\QQ\abs\PP$, it holds that $\onedim\cH^1(\QQ)\leq\onedim\cH^1(\PP)$. 
If furthermore $\QQ\equ\PP$, then $\onedim\cH^1(\PP)=\onedim\cH^1(\QQ)$ and an $\R^d$-valued local martingale $X$ on $(\Omega,\cF,\FF,\PP)$ is a basis for $\cH^1(\PP)$ if and only if $\widehat{X}$ is a basis for $\cH^1(\QQ)$.
\end{prop}

This last result generalizes \cite[Theorem 3.2 and its Corollary]{Duffie85} by removing all restrictive boundedness assumptions on the density process $Z$.

\section{Examples}	\label{sec:example}

\subsection{An example of MRP when $\boldsymbol{[X,Z]\notin}\Aloc\boldsymbol{(\PP)}$}	\label{sec:example1}

In this subsection, we present an example of a simple situation where classical results on the stability of the MRP under absolutely continuous changes of probability cannot be applied, while on the contrary our Theorem \ref{thm2} yields the explicit existence of a process having the MRP.

On a probability space $(\Omega,\cF,\PP)$, let $N=(N_t)_{t\geq0}$ be a standard Poisson process with intensity $1$ on its natural filtration $\FF=(\cF_t)_{t\geq0}$ and denote by $M=(M_t)_{t\geq0}$ the associated compensated martingale, i.e. $M_t:=N_t-t$, for all $t\geq0$. It is well-known that $M$ has the MRP under $\PP$ (see, e.g., \cite[Proposition 8.3.5.1]{JYC09}).
Let $\tau_1:=\inf\{t\in\R_+ : N_t>0\}$ be the first jump time of $N$. By \cite[Lemma 13.8]{HWY}, the stopped martingale $M^{\tau_1}$ has the MRP on $(\Omega,\cF,\FF^{\tau_1},\PP)$, where $\FF^{\tau_1}$ denotes the stopped filtration $(\cF_{t\wedge \tau_1})_{t\geq0}$.
We then define the process $X=(X_t)_{t\geq0}$ by
\[
X_t := \int_0^t\frac{1}{\sqrt{u}}\,\ud M^{\tau_1}_u,
\qquad\text{ for all }t\geq0.
\]
It holds that $X\in\cH^1_0(\PP)$, as follows from the fact that
\[
\EE\bigl[[X]_{\infty}^{1/2}\bigr]
= \EE\left[\left(\int_0^{\infty}\!\frac{1}{u}\,\ud[M]^{\tau_1}_u\right)^{1/2}\right]
= \EE\left[\frac{1}{\sqrt{\tau_1}}\right]
= \int_0^{+\infty}\frac{e^{-u}}{\sqrt{u}}\,\ud u = \sqrt{\pi} <+\infty.
\]
Moreover, since the integrand $1/\sqrt{u}$ is strictly positive, it is immediate to verify that the martingale $X$ inherits the MRP of $M^{\tau_1}$ under $\PP$ in the filtration $\FF^{\tau_1}$.

For a constant $T\in(0,1/4]$, let us define the uniformly integrable martingale $Z:=1+X^T$ on $(\Omega,\cF,\FF^{\tau_1},\PP)$. Note that
\[
Z_{\infty} 
= 1+X_T
= 1+ \ind_{\{\tau_1\leq T\}}\left(\frac{1}{\sqrt{\tau_1}}-2\sqrt{\tau_1}\right) - \ind_{\{\tau_1>T\}}2\sqrt{T} \geq0,
\]
with $\PP(Z_{\infty}=0)=\PP(\tau_1>1/4)>0$ holding for $T=1/4$. We can therefore define the probability measure $\QQ\ll\PP$ by $\ud\QQ:=Z_{\infty}\ud\PP$. Note that $\QQ\sim\PP$ holds if and only if $T<1/4$.

In the context of the present example, existing results (such as \cite[Theorem 3.2]{Duffie85}, \cite[Theorem 13.12]{HWY}, \cite[Theorem 15.2.8]{CohenElliott} or \cite[Lemma 2.5]{JS15}) cannot be applied to deduce the existence of a process having the MRP under $\QQ$, since the process $[X,Z]$ fails to be locally integrable under $\PP$. 
Indeed, for every $t\in(0,T]$, it holds that
\[
\EE\bigl[[X,Z]_t\bigr]
= \EE\bigl[[X]_t\bigr]
= \EE\left[\int_0^t\frac{1}{u}\ud[M]^{\tau_1}_u\right]
= \EE\left[\frac{1}{\tau_1}\ind_{\{\tau_1\leq t\}}\right]
= \int_0^t\frac{e^{-u}}{u}\ud u = +\infty,
\]
which in turn implies that $\EE[[X,Z]_{\sigma}]=+\infty$ for every stopping time $\sigma$ with $\PP(\sigma>0)>0$.

However, as a consequence of Theorem \ref{thm2}, the process $\widehat{X}$ has the MRP under $\QQ$ and, in view of Proposition \ref{prop:Leng}-(iii), it can be explicitly computed as follows.
Note that $\eta=+\infty$ for all $T\in(0,1/4]$, so that $\Delta X_{\eta}\ind_{\{\eta<+\infty\}}=0$. Therefore, for all $t\geq0$, it holds that
\begin{align*}
\widehat{X}_t 
&= X_t - \int_0^t\frac{1}{Z_u}\ud[X,Z]_u
= X_t - \frac{1}{\tau_1Z_{\tau_1}}\ind_{\{\tau_1\leq t\wedge T\}}	\\
&= \ind_{\{\tau_1\leq t\}}\left(\frac{1}{\sqrt{\tau_1}}-2\sqrt{\tau_1}\right) - \ind_{\{\tau_1>t\}}2\sqrt{t}
- \frac{1}{\sqrt{\tau_1}\left(1+\sqrt{\tau_1}-2\tau_1\right)}\ind_{\{\tau_1\leq t\wedge T\}}.
\end{align*}

\subsection{Further examples}		\label{sec:further_examples}

In this section, we present two counterexamples related to the notions of strong orthogonality and orthogonality in the usual sense of local martingales, which appear in the context of Proposition \ref{prop:orth}.

\begin{exa}		\label{ce:stong}
On a probability space $(\Omega,\cF,\PP)$, consider two random variables $\varepsilon$ and $\xi$ of the form
\[
\varepsilon = \begin{cases}
1, & \text{ with probability }p\in(0,1),\\
0, & \text{ with probability }1-p,
\end{cases}
\qquad\text{ and }\qquad
\xi = \begin{cases}
+1, & \text{ with probability }1/2,\\
-1, & \text{ with probability }1/2,
\end{cases}
\]
and such that $\varepsilon$ and $\xi$ are independent.
Define the processes $X$ and $Y$ by 
\[
X := 1 + \varepsilon\xi\ind_{\dbraco{1,+\infty}}
\qquad\text{ and }\qquad
Y := 1 + (1-\varepsilon)\xi\ind_{\dbraco{1,+\infty}},
\]
and let $\FF$ be the associated natural filtration.
Clearly, $X$ and $Y$ are martingales on $(\Omega,\cF,\FF,\PP)$. Moreover, $X$ and $Y$ are strongly orthogonal under $\PP$, since
\[
[X,Y] 
= \sum_{0<s\leq\cdot} \Delta X_s\Delta Y_s
= \varepsilon(1-\varepsilon)\ind_{\dbraco{1,+\infty}} = 0.
\]
Define the probability measure $\QQ\ll\PP$ by $\ud\QQ = X_1\ud\PP$, with density process $X$. In this case,
\[
\eta = \ind_{\{\varepsilon=1,\xi=-1\}} + \infty\ind_{\{\varepsilon=0\}\cup\{\xi=+1\}},
\]
which is an accessible time (i.e., $\eta=\eta^{{\rm ac}}$).
Moreover, we have that
\[
A^X := (\Delta X_{\eta}\ind_{\dbraco{\eta,+\infty}})^{p,\PP}
= -(\ind_{\{\varepsilon=1,\xi=-1\}}\ind_{\dbraco{1,+\infty}})^{p,\PP}
= -\PP(\varepsilon=1,\xi=-1|\cF_{1-})\ind_{\dbraco{1,+\infty}}
= -\frac{p}{2}\ind_{\dbraco{1,+\infty}}.
\]
Therefore, it holds that
\[
\Delta Y\Delta A^X
= -\frac{p}{2}(1-\varepsilon)\xi\ind_{\dbra{1}}.
\]
In particular, since $\Delta A^Y=0$ $\PP$-a.s., this implies that $[\widehat{X},\widehat{Y}]\neq0$, thus showing that $\widehat{X}$ and $\widehat{Y}$ are not strongly orthogonal under $\QQ$.
Observe that, in this example, the condition $\Delta X_{\eta^{{\rm ac}}}=0$ $\PP$-a.s. on $\{\eta^{\rm ac}<+\infty\}$ fails to hold, thus showing its necessity in the statement of Proposition \ref{prop:orth}.
\end{exa}

\begin{exa}	\label{example:usual_orth}
In this example, we construct a bi-dimensional local martingale $(X,Y)$ on a filtered probability space $(\Omega,\cF,\FF,\PP)$ such that $XY\in\MlocP$, while $\widehat{X}\widehat{Y}\notin\MlocQ$ for some probability measure $\QQ\sim\PP$. In particular, this shows that one cannot obtain a version of Proposition \ref{prop:orth} for the usual notion of orthogonality in the case of general discontinuous local martingales.

On a filtered probability space $(\Omega,\cF,\FF,\PP)$, let $W=(W_t)_{t\geq0}$ be a standard Brownian motion and $M=(M_t)_{t\geq0}$ a compensated Poisson process with intensity 1.
Define the two processes
\[
X := \cE(W+M)
\qquad\text{ and }\qquad
Y:= \cE(W-M),
\]
which are martingales on $(\Omega,\cF,\FF,\PP)$ and admit explicit solutions
\[
X_t = e^{W_t-\frac{3}{2}t}2^{N_t}
\qquad\text{ and }\qquad
Y_t = e^{W_t+\frac{t}{2}}\ind_{\{t<\tau_1\}},
\qquad\text{ for all }t\geq0,
\]
where $\tau=\inf\{t\in\R_+ : \Delta M_t\neq0\}$.
Note that $\Delta X_{\tau}=X_{\tau-}$ and $\Delta Y_{\tau}=-Y_{\tau-}$. 
The $\PP$-martingales $X$ and $Y$ are orthogonal (in the usual sense, but not strongly orthogonal), indeed:
\[
[X,Y] = X_-Y_-\cdot[W+M,W-M]  = -X_-Y_-\cdot M \in \MlocP.
\]
Define now the probability measure $\QQ\sim\PP$ by $\ud\QQ = X_{\tau}\ud\PP$, with density process $X^{\tau}$. 
The $\QQ$-local martingales $\widehat{X}$ and $\widehat{Y}$ are orthogonal under $\QQ$ (in the usual sense of local martingales) if and only if $V:=[\widehat{X},\widehat{Y}]X^{\tau}\in\MlocP$.
By integration by parts and Yoeurp's lemma, denoting by $\underset{{\rm loc.mart.}}{=}$ equality up to a $\PP$-local martingale term which may change from line to line, it holds that
\begin{align*}
V &\underset{{\rm loc.mart.}}{=} 
X^{\tau}_-\cdot[\widehat{X},\widehat{Y}] + \bigl[X,[\widehat{X},\widehat{Y}]\bigr]^{\tau}
= X\cdot[\widehat{X},\widehat{Y}]^{\tau}
= X\cdot\left[X-\frac{1}{X}\cdot[X],Y-\frac{1}{X}\cdot[X,Y]\right]^{\tau}	\\
&\underset{\phantom{{\rm loc.mart.}}}{=} 
X\cdot[X,Y]^{\tau} - 2(\Delta X_{\tau})^2\Delta Y_{\tau}\ind_{\dbraco{\tau,+\infty}} + \frac{1}{X_{\tau}}\Delta Y_{\tau}(\Delta X_{\tau})^3\ind_{\dbraco{\tau,+\infty}}	\\
&\underset{{\rm loc.mart.}}{=} 
(\Delta X_{\tau})^2\Delta Y_{\tau}\left(\frac{\Delta X_{\tau}}{X_{\tau}}-1\right)\ind_{\dbraco{\tau,+\infty}}
= \frac{1}{2}(X_{\tau-})^2Y_{\tau-}\ind_{\dbraco{\tau,+\infty}}.
\end{align*}
Since $X_{\tau-}Y_{\tau-}>0$ a.s., this shows that the process $V$ cannot be a local martingale under $\PP$, so that $\widehat{X}$ and $\widehat{Y}$ are not orthogonal under $\QQ$ in the usual sense of local martingales.
\end{exa}

\section{Proofs}		\label{sec:proofs}

In this section, we give the proofs of our results, together with some auxiliary technical results.
As a preliminary, let us define the probability measure $\overline{\PP}:=(\PP+\QQ)/2$ on $(\Omega,\cF)$ and denote by $\FF^{\overline{\PP}}$ the $\overline{\PP}$-completion of the filtration $\FF$. 
For $R\in\{\PP,\QQ\}$, if $X$ is an $\FF$-adapted $R$-a.s. c\`adl\`ag process, then $X$ is $R$-indistinguishable from an $\FF^{\overline{\PP}}$-optional process. Therefore, without loss of generality, we shall take $\FF^{\overline{\PP}}$-optional versions of all relevant processes.
We then recall \cite[Lemma 7.21]{J79}, which ensures that path properties are preserved under locally absolutely continuous changes of probability.

\begin{lem}	\label{lem:RC}
Let $X$ be an $\FF^{\overline{\PP}}$-optional process. Then the following are equivalent:
\begin{enumerate}
\item[(i)] $X$ has $\PP$-a.s. c\`adl\`ag paths on $\dbraco{0,\zeta}$;
\item[(ii)] $X$ has $\QQ$-a.s. c\`adl\`ag paths.
\end{enumerate}
\end{lem}

In particular, in the context of Proposition \ref{prop:Leng}, Lemma \ref{lem:RC} ensures that an adapted process $X$ has $\QQ$-a.s. c\`adl\`ag paths if and only if the process $ZX$ has $\PP$-a.s. c\`adl\`ag paths.

\subsection{Proofs of the results stated in Section \ref{sec:main}}
\label{sec:proofs_main}

We recall that, for $M\in\MlocP$, the process $\widehat{M}$ is the Lenglart transformation of $M$, i.e., the element of $\MlocQ$ defined via \eqref{eq:transform}.

\begin{proof}[Proof of Theorem \ref{thm1}]
In view of \cite[Corollary 4.12]{J79}, in order to prove $\cH^1_0(\QQ)=\cL^1(\widehat{\MP},\QQ)$, it suffices to show that every  bounded $N\in\MQzero$ such that $N\widehat{M}\in\MlocQ$, for all $M\in\MP$, is null.
Recalling that $\QQ(\zeta<+\infty)=0$, we can apply integration by parts under $\QQ$ and compute
\begin{align}
N = \frac{ZN}{Z}
&= \frac{1}{Z_-}\cdot (ZN) + (ZN)_-\cdot\frac{1}{Z} + \left[ZN,\frac{1}{Z}\right]	\notag\\
&= \frac{1}{Z_-}\cdot\left(ZN - \frac{1}{Z}\cdot[ZN,Z]\right) - \frac{(ZN)_-}{Z^2_-}\cdot\left(Z-\frac{1}{Z}\cdot[Z]\right),
\label{eq:int_parts}
\end{align}
where the last equality makes use of the identities
\be	\label{eq:computations}
\left[ZN,\frac{1}{Z}\right] = - \frac{1}{ZZ_-}\cdot[ZN,Z]
\qquad\text{ and }\qquad
\frac{1}{Z} = \frac{1}{Z_0} - \frac{1}{Z^2_-}\cdot\left(Z-\frac{1}{Z}\cdot[Z]\right),
\ee
as can be readily verified by applying It\^o's formula (under $\QQ$), see also \eqref{eq:computeZ} below.
Furthermore, on $\{\eta<+\infty\}$ (under the measure $\PP$) it holds that
$
\Delta (ZN)_{\eta}
= \frac{(ZN)_{\eta-}}{Z_{\eta-}}\Delta Z_{\eta},
$
which implies that $(\Delta(ZN)_{\eta}\ind_{\dbraco{\eta,+\infty}})^{p,\PP}=((ZN)_-/Z_-)\cdot(\Delta Z_{\eta}\ind_{\dbraco{\eta,+\infty}})^{p,\PP}$.
In turn, making use of representation \eqref{eq:transform}, this enables us to rewrite \eqref{eq:int_parts} under $\QQ$ as follows:
\[
N = \frac{1}{Z_-}\cdot\widehat{ZN} - \frac{N_-}{Z_-}\cdot\widehat{Z}.
\]
Since $N$ is bounded, $[N]\in\Aloc(\QQ)$ and therefore the predictable quadratic variation $\langle N\rangle^{\QQ}$ of $N$ under $\QQ$ is well-defined and can be explicitly computed as
\[
\langle N\rangle^{\QQ} = \frac{1}{Z_-}\cdot\langle N,\widehat{ZN}\rangle^{\QQ} - \frac{N_-}{Z_-}\cdot\langle N,\widehat{Z}\rangle^{\QQ}.
\]
By assumption, $N\widehat{M}\in\MlocQ$ for every $M\in\MP$.
Since $ZN$ and $Z$ belong to $\MP$, it follows that $\langle N,\widehat{ZN}\rangle^{\QQ}\equiv0$ and $\langle N,\widehat{Z}\rangle^{\QQ}\equiv0$. Therefore, we have that $\langle N\rangle^{\QQ}\equiv0$, thus proving that $N$ is null (up to an evanescent set under $\QQ$).
Finally, $\MlocQzero=\cL_{{\rm loc}}^1(\widehat{\MP},\QQ)$ is a direct consequence of the fact that $\MlocQ=\cH^1_{{\rm loc}}(\QQ)$, see \cite[Proposition 2.38]{J79}.
\end{proof}

We proceed to proving our second main result (Theorem \ref{thm2}). This makes use of the following two technical lemmata, which concern the behavior of continuous and purely discontinuous local martingales and stochastic integrals under locally absolutely continuous changes of probability. 
We recall that, for two semimartingales $X$ and $Y$ on $(\Omega,\cF,\FF,\PP)$, the quadratic variation $[X,Y]$ under $\PP$ is also a version of the quadratic variation under $\QQ$ (see \cite[Theorem III.3.13]{JS03}).


\begin{lem}	\label{lem:cont_disc}
If $M\in\mMloc^c(\PP)$ ($\mMloc^d(\PP)$, resp.), then $\widehat{M}\in\mMloc^c(\QQ)$ ($\mMloc^d(\QQ)$, resp.).
\end{lem}
\begin{proof}
Consider first $M\in\Mloc^c(\PP)$. Proposition \ref{prop:Leng}-(iii) implies that $\widehat{M}=M-(1/Z)\cdot[M,Z]$. Since $[M,Z]=[M,Z]^c$, it immediately follows that $\widehat{M}\in\Mloc^c(\QQ)$.
Consider then $M\in\Mloc^d(\PP)$. In this case, setting $A:= (\Delta M_{\eta}\ind_{\dbraco{\eta,+\infty}})^{p,\PP}$ for brevity of notation, we compute
\begin{align*}
[\widehat{M}]
&= [M] + \left[\frac{1}{Z}\cdot[M,Z]\right] + [A]
- 2\left[M,\frac{1}{Z}\cdot[M,Z]\right] 
+ 2\left[M,A\right] - 2\left[\frac{1}{Z}\cdot[M,Z],A\right]	\\
&= \sum_{s\leq\cdot}(\Delta M_s)^2 
+ \sum_{s\leq\cdot}\frac{(\Delta M_s\Delta Z_s)^2}{Z^2_s} 
+ \sum_{s\leq\cdot}(\Delta A_s)^2
- 2\sum_{s\leq\cdot}\frac{(\Delta M_s)^2\Delta Z_s}{Z_s}	\\
&\quad\;
+ 2\sum_{s\leq\cdot}\Delta M_s\Delta A_s
- 2\sum_{s\leq\cdot}\frac{\Delta M_s\Delta Z_s\Delta A_s}{Z_s}
= \sum_{s\leq\cdot}(\Delta\widehat{M}_s)^2.
\end{align*}
We have thus shown that $[\widehat{M}]=\sum_{s\leq\cdot}(\Delta\widehat{M}_s)^2$, which means that $\widehat{M}\in\Mloc^d(\QQ)$. 
\end{proof}

\begin{lem}	\label{lem:stoch_int}
If $M$ is an $\R^d$-valued process such that $M^i\in\mMlocP$, for each $i=1,\ldots,d$, and $H\in L_m(M,\PP)$, then $H\in L_m(\widehat{M},\QQ)$ and $\widehat{H\cdot M}$ is a version of $H\cdot\widehat{M}$.
\end{lem}
\begin{proof}
For simplicity of presentation, we only prove the claim for a real-valued process $M$, the multi-dimensional case being analogous.
It suffices to consider separately the cases $M\in\Mloc^c(\PP)$ and $M\in\Mloc^d(\PP)$. 
If $M\in\Mloc^c(\PP)$, then the claim follows from \cite[Proposition 7.26]{J79}, since in the continuous case the Lenglart transformation coincides with the usual Girsanov transformation.


Consider then the case $M\in\Mloc^d(\PP)$. By Lemma \ref{lem:cont_disc}, $\widehat{M}$ and $\widehat{H\cdot M}$ belong to $\Mloc^d(\QQ)$. Therefore, in view of \cite[Definition 2.46]{J79}, in order to show that $\widehat{H\cdot M}$ is a version of $H\cdot\widehat{M}$ it suffices to show that $\Delta(\widehat{H\cdot M})=H\Delta\widehat{M}$. To this end, making use of formula \eqref{eq:transform}, we compute
\be	\label{eq:lem_si_proof}
\Delta(\widehat{H\cdot M}) = \Delta(H\cdot M) - \frac{\Delta(H\cdot M)\Delta Z}{Z} + \Delta\bigl(\Delta(H\cdot M)_{\eta}\ind_{\dbraco{\eta,+\infty}}\bigr)^{p,\PP}.
\ee
Note that $\Delta(H\cdot M)=H\Delta M$ and, due to the predictability of $H$, it holds that
\[
\bigl(\Delta(H\cdot M)_{\eta}\ind_{\dbraco{\eta,+\infty}}\bigr)^{p,\PP}
= \bigl(H_{\eta}\Delta M_{\eta}\ind_{\dbraco{\eta,+\infty}}\bigr)^{p,\PP}
= \bigl(H\cdot(\Delta M_{\eta}\ind_{\dbraco{\eta,+\infty}})\bigr)^{p,\PP}
= H\cdot\bigl(\Delta M_{\eta}\ind_{\dbraco{\eta,+\infty}}\bigr)^{p,\PP}.
\]
This enables us to rewrite \eqref{eq:lem_si_proof} as follows:
\[
\Delta(\widehat{H\cdot M})
= H\Delta M - \frac{H\Delta M\Delta Z}{Z} + H\Delta\bigl(\Delta M_{\eta}\ind_{\dbraco{\eta,+\infty}}\bigr)^{p,\PP}
= H\Delta\widehat{M},
\]
thus completing the proof.
\end{proof}

\begin{proof}[Proof of Theorem \ref{thm2}]
Under the present assumptions, it holds that
\[
\widehat{\MPzero} = \bigl\{\widehat{H\cdot X} : H\in L_m(X,\PP)\text{ and }H\cdot X\in\MP\bigr\}
\subseteq 
\bigl\{H\cdot\widehat{X} : H\in L_m(\widehat{X},\QQ)\bigr\},
\]
where the last inclusion follows from Lemma \ref{lem:stoch_int}.
By Theorem \ref{thm1} we then have
\[
\MlocQzero 
\subseteq \cLloc^1\bigl(\bigl\{H\cdot\widehat{X} : H\in L_m(\widehat{X},\QQ)\bigr\},\QQ\bigr)
= \cLloc^1(\widehat{X},\QQ)
= \bigl\{H\cdot\widehat{X} : H\in L_m(\widehat{X},\QQ)\bigr\},
\]
where the first equality follows by the associativity of the stochastic integral and the second from \cite[Theorem 4.6]{J79}.
Since $\bigl\{H\cdot\widehat{X} : H\in L_m(\widehat{X},\QQ)\bigr\}
\subseteq \MlocQzero$, the theorem is proved.
\end{proof}

\subsection{Proof of the representation stated in Remark \ref{rem:integrand}}	\label{sec:proof_rem}

In this subsection, we prove the explicit representation given in equation \eqref{eq:integrand}. We proceed by adapting to the present setting some of the arguments used in the proof of \cite[Theorem 2.6]{F18}.

Since $\QQ\abs\PP$, the density process $Z$ is a strictly positive semimartingale under $\QQ$ (see \cite[Theorem III.3.13]{JS03}). Therefore, as in \eqref{eq:computations}, an application of It\^o's formula yields that
\be	\label{eq:computeZ}
\frac{1}{Z} 
= \frac{1}{Z_0} - \frac{1}{Z^2_-}\cdot Z + \frac{1}{Z^3_-}\cdot[Z]^c + \sum_{s\leq\cdot}\left(\frac{1}{Z_s}-\frac{1}{Z_{s-}}+\frac{\Delta Z_s}{Z^2_{s-}}\right)
= \frac{1}{Z_0} - \frac{1}{Z^2_-}\cdot\left(Z-\frac{1}{Z}\cdot[Z]\right).
\ee
By MRP under $\PP$, there exists a process $H\in L_m(X,\PP)$ such that $Z=Z_0+H\cdot X$ (see equation \eqref{eq:MRP_Z}). In view of \cite[Proposition III.6.24]{JS03}, the process $H$ is integrable with respect to $X$ under $\QQ$ in the semimartingale sense. Hence, by the associativity of the stochastic integral, we have that
\be	\label{eq:1/Z}
\frac{1}{Z} = \frac{1}{Z_0} - \frac{H}{Z^2_-}\cdot\left(X-\frac{1}{Z}\cdot[X,Z]\right).
\ee
Recall from  \eqref{eq:MRP_Z} that there exists a sequence of stopping times $(\tau_n)_{n\in\N}$ increasing $\QQ$-a.s. to infinity such that $(ZN)^{\tau_n}=K^n\cdot X$, with $K^n\in L_m(X,\PP)$, for each $n\in\N$. 
Similarly as above, $K^n\cdot X$ also makes sense as a semimartingale stochastic integral under $\QQ$. Therefore, using similar arguments as in the proof of Theorem \ref{thm1}, we can apply integration by parts and equation \eqref{eq:1/Z}, thus obtaining
\begin{align}
N^{\tau_n} = \frac{(ZN)^{\tau_n}}{Z^{\tau_n}}
&= (ZN)_-\cdot\frac{1}{Z^{\tau_n}} + \frac{1}{Z_-}\cdot(ZN)^{\tau_n} + \left[\frac{1}{Z},ZN\right]^{\tau_n}	\notag\\
&= -\frac{(ZN)_-H}{Z^2_-}\cdot\left(X^{\tau_n}-\frac{1}{Z}\cdot[X,Z]^{\tau_n}\right) + \frac{K^n}{Z_-}\cdot X^{\tau_n} - \frac{K^n}{ZZ_-}\cdot[X,Z]^{\tau_n}	\notag\\
&= \phi^n\cdot\left(X^{\tau_n}-\frac{1}{Z}\cdot[X,Z]^{\tau_n}\right),
\label{eq:int_parts_MRP}
\end{align}
where, for each $n\in\N$,
\[
\phi^n := \frac{1}{Z_-}\left(K^n-\frac{(ZN)_-H}{Z_-}\right)\ind_{\dbra{0,\tau_n}}.
\]
Moreover, on $\{\eta\leq\tau_n\}\cap\{\eta<+\infty\}$ (under the measure $\PP$) it holds that
\[
\left(K^n_{\eta} - \frac{(ZN)_{\eta-}H_{\eta}}{Z_{\eta-}}\right)\Delta X^{\tau_n}_{\eta}
= \Delta(ZN)^{\tau_n}_{\eta} - \frac{(ZN)_{\eta-}}{Z_{\eta-}}\Delta Z^{\tau_n}_{\eta}
= 0.
\]
In turn, this implies that 
$
0=(\phi^n_{\eta}\Delta X^{\tau_n}_{\eta}\ind_{\dbraco{\eta,+\infty}})^{p,\PP}
= \phi^n\ind_{\{Z_->0\}}\cdot(\Delta X^{\tau_n}_{\eta}\ind_{\dbraco{\eta,+\infty}})^{p,\PP}
$
up to an evanescent set. Therefore, by \eqref{eq:int_parts_MRP} together with \eqref{eq:transform}, it follows that, up to a $\QQ$-evanescent set,
\[
N^{\tau_n} = \phi^n\cdot\widehat{X}^{\tau_n},
\qquad\text{ for each }n\in\N.
\]
Setting $\phi:=\sum_{n\in\N}\phi^n\ind_{\dbraoc{\tau_{n-1},\tau_n}}$, with $\tau_0:=0$, we finally obtain that $N=\phi\cdot\widehat{X}$.

\subsection{Proof of the results stated in Section \ref{sec:corollaries}}

In this section, we present the proof of the remaining results of the paper, starting with Proposition \ref{prop:orth}.

\begin{proof}[Proof of Proposition \ref{prop:orth}]
We have to show that, if $\Delta X_{\eta^{{\rm ac}}}=0$ $\PP$-a.s. on $\{\eta^{{\rm ac}}<+\infty\}$ and $[X^i,X^j]\equiv0$, for all $i,j=1,\ldots,d$ with $i\neq j$, then $[\widehat{X}^i,\widehat{X}^j]\equiv0$ (up to a $\QQ$-evanescent set), for all $i,j=1,\ldots,d$ with $i\neq j$.
We first compute
\[
\bigl[(\widehat{X}^i)^c,(\widehat{X}^j)^c\bigr]
= \bigl[\widehat{(X^i)^c},\widehat{(X^j)^c}\bigr] 
= \bigl[(X^i)^c,(X^j)^c\bigr],
\]
where the first equality follows from Lemma \ref{lem:cont_disc} and the uniqueness of the decomposition of a local martingale into a continuous part and a purely discontinuous part (see \cite[Theorem I.4.18]{JS03}).
The assumption that $[X^i,X^j]\equiv0$ implies that $[(X^i)^c,(X^j)^c]\equiv0$, so that $[(\widehat{X}^i)^c,(\widehat{X}^j)^c]\equiv0$ up to a $\QQ$-evanescent set.
It remains to show that $\Delta\widehat{X}^i_t\,\Delta\widehat{X}^j_t=0$ $\QQ$-a.s. for all $t\geq0$ and $i,j=1,\ldots,d$ with $i\neq j$.
For brevity of notation, let $A^i:=(\Delta X^i_{\eta}\ind_{\dbraco{\eta,+\infty}})^{p,\PP}$, for $i=1,\ldots,d$.
If $\Delta X_{\eta^{{\rm ac}}}=0$ a.s. on $\{\eta^{{\rm ac}}<+\infty\}$, it holds that $A^i=(\Delta X^i_{\eta^{{\rm in}}}\ind_{\dbraco{\eta^{{\rm in}},+\infty}})^{p,\PP}$. Since the process $\Delta X^i_{\eta^{{\rm in}}}\ind_{\dbraco{\eta^{{\rm in}},+\infty}}$ is quasi-left-continuous, \cite[Proposition I.2.35]{JS03} implies that $A^i$ is continuous $\PP$-a.s. and, hence, $\QQ$-a.s. 
(see \cite[Theorem 12.9]{HWY}).
In view of equation \eqref{eq:transform}, we therefore obtain that
\[
\Delta\widehat{X}^i = \Delta X^i - \frac{\Delta X^i\Delta Z}{Z}.
\]
Recall that $[X^i,X^j]\equiv0$ implies that $\Delta X^i_t\,\Delta X^j_t=0$ $\PP$-a.s. for all $t\geq0$. Since by assumption $\QQ\abs\PP$, we deduce that $\Delta\widehat{X}^i_t\,\Delta\widehat{X}^j_t=0$.
\end{proof}

We proceed with the proof of Proposition \ref{prop:equiv}, which relies on the symmetric role of the two probabilities $\QQ$ and $\PP$ under the assumption $\QQ\equ\PP$.

\begin{proof}[Proof of Proposition \ref{prop:equiv}]
By Theorem \ref{thm2}, it suffices to show that, if $\widehat{X}$ has the MRP under $\QQ$, then $X$ has the MRP under $\PP$.
Since $\QQ\sim_{{\rm loc}}\PP$, the density process of $\PP$ relative to $\QQ$ is given by $1/Z$. 
We can then apply Proposition \ref{prop:Leng}-(iii) on $(\Omega,\cF,\FF,\QQ)$ to the process $\widehat{X}$, yielding
\begin{align*}
\widehat{X} - Z\cdot\left[\widehat{X},\frac{1}{Z}\right]
&= X - \frac{1}{Z}\cdot[X,Z] - Z\cdot\left[X,\frac{1}{Z}\right] + Z\cdot\left[\frac{1}{Z}\cdot[X,Z],\frac{1}{Z}\right]	\\
&= X - \frac{1}{Z}\cdot[X,Z] + \frac{1}{Z_-}\cdot[X,Z] - \sum_{s\leq\cdot} \frac{\Delta X(\Delta Z)^2}{ZZ_-}
= X.
\end{align*}
As a consequence of Theorem \ref{thm2} applied to the local martingale $\widehat{X}$ on $(\Omega,\cF,\FF,\QQ)$, the process $X$ has the MRP under $\PP$, thus proving the claim.
\end{proof}

Finally, we conclude with the short proof of Proposition \ref{prop:dim}.

\begin{proof}[Proof of Proposition \ref{prop:dim}]
It suffices to consider the case $\onedim\cH^1(\PP)<+\infty$.
Let $X$ be an $\R^d$-valued local martingale on $(\Omega,\cF,\FF,\PP)$ that is a basis for $\cH^1(\PP)$ and let $\widehat{X}$ be defined as in \eqref{eq:transform}. By Theorem \ref{thm2}, it holds that $\cL^1(\widehat{X},\QQ)=\cH^1_0(\QQ)$. 
This implies that $\onedim\cH^1(\QQ)\leq d=\onedim\cH^1(\PP)$. If $\QQ\equ\PP$, using the result of Proposition \ref{prop:equiv} and reversing the role of $\PP$ and $\QQ$ in the previous argument, we obtain that $\onedim\cH^1(\PP)\leq \onedim\cH^1(\QQ)$, thus proving the claim.
\end{proof}

\bibliographystyle{alpha}
\bibliography{abs_cont}

\end{document}